\documentclass[12pt,reqno]{amsart}

\usepackage[colorlinks=false]{hyperref}
\usepackage{graphicx}
\usepackage[dvipsnames]{xcolor}
\usepackage{rotating}
\usepackage{amssymb}
\usepackage{amsmath,amsfonts}
\usepackage[utf8]{inputenc}
\usepackage{verbatim}
\usepackage[tableposition=top,font=small,labelfont=bf,format=hang]{caption}
\usepackage{booktabs}

\newtheorem{lem}{Lemma}
\newtheorem{thm}{Theorem}
\newtheorem{prop}{Proposition}

\newtheorem{rem}{Remark}
\newtheorem{coro}{Corollary}
\newtheorem{ex}{Example}

\newcommand{\F}{\mathbb{F}}
\newcommand{\Z}{\mathbb{Z}}
\newcommand{\Q}{\mathbb{Q}}

\newcommand{\ord}{\operatorname{ord}}
\newcommand{\Tr}{\operatorname{Tr}}
\newcommand{\WSS}{\mathrm{WSS}}

\DeclareMathOperator{\wt}{wt}
\renewcommand{\vec}[1]{\boldsymbol{#1}} 

\title[Second-order recurrences and cyclic codes]{Second-order recurrences, quadratic number fields and cyclic codes}
\thanks{The work of Minjia Shi was supported in part by the National Natural Science Foundation of China under Grant 12471490.
The work of Jon-Lark Kim was supported in part by the BK21 FOUR (Fostering Outstanding Universities for Research) funded by the Ministry of
Education(MOE, Korea) and National Research Foundation of Korea(NRF)
under Grant No. 4120240415042 and by Basic Science Research Program through the National Research Foundation of Korea(NRF) funded by the Ministry of Science and ICT under Grant No. RS-2025-24534992.}
\author{Minjia Shi}
\address{Key Laboratory of Intelligent Computing Signal Processing, Ministry of Education,
School of Mathematical Sciences, Anhui University, Hefei 230601, China}
\email{smjwcl.good@163.com}
\author{Xuan Wang}
\address{Key Laboratory of Intelligent Computing Signal Processing, Ministry of Education,
School of Mathematical Sciences, Anhui University, Hefei 230601, China}
\email{wang\_xuan\_ah@163.com}
\author{Zakariae Bouazzaoui}
\address{School of Education and Training, University Mohammed I, Oujda, Morocco}
\email{z.bouazzaoui@ump.ac.ma}
\author{Jon-Lark Kim}
\address{Department of Mathematics and Institute for Mathematical and Data Sciences, Sogang University, Seoul, South Korea}
\email{jlkim@sogang.ac.kr}
\author{Patrick Sol\'e}
\address{I2M, Aix Marseille Univ, CNRS,  Marseille, France}
\email{patrick.sole@telecom-paris.fr}

\date{}

\begin{document}
\begin{abstract}
	Wall--Sun--Sun primes are the primes $p$ for which the Fibonacci recurrence has the same period modulo $p$ and modulo $p^2$.  
	We study the analogous equality for invertible second-order recurrences.  
	We distinguish the property attached to a specified quadratic polynomial from the existential property attached only to the square-free part of its discriminant.  
	For recurrences arising from fundamental units, and under the standard arithmetic hypotheses, equality of the two periods is equivalent to failure of $p$-rationality of the associated real quadratic field.  
	We associate to such recurrences free cyclic codes of dimension two over $\F_p$ and $\Z_{p^2}$ and determine their Hamming weight distributions in several reducible and irreducible families.  
	The reducible construction gives MDS codes.  
	The irreducible construction uses an element of exact order $2(p+1)$ in $\F_{p^2}^{\times}$ and, when $p\equiv3\pmod4$, gives an extremal NMDS code with parameters $[2p+2,2,2p]_p$.
\end{abstract}
\maketitle

\section{Introduction}

The famous Fibonacci sequence is defined by the linear recurrence
$$F_{n}=F_{n-1}+F_{n-2},\; n\ge 2,$$
with initial conditions $F_0=0,\,F_1=1.$
This is a periodic sequence modulo every positive integer $m$, its period $\pi(m)$ is called the Pisano period.
Formulas and properties of $\pi(m)$ are well known since Wall's paper in $1960$ \cite{W}. That paper renewed interest in the subject
and motivated studies of Pisano periods, which were later related to other objects and generalized to other classes of second order recurrence
sequences. A natural question was asked by Wall in his paper: Can there be a prime $p$ such that $\pi(p)=\pi(p^2)?$
It is known that up to $10^{14}$, there are no such primes (cf. \cite{MR}). Still, using heuristics and probabilistic arguments, some authors conjecture the existence of infinitely many primes $p$ satisfying $\pi(p)=\pi(p^2)$ \cite{EJ,Kl}. It was shown by the Sun brothers in \cite{SS} that the existence of these primes is related to Fermat's Last Theorem, the reason why we call them Wall-Sun-Sun primes, or WSS primes for short. 

The question of Wall can naturally be asked for other second-order
recurrences.  Let
\[
f(X)=X^2+AX+B\in\Z[X],\qquad p\nmid B,
\]
and let $k_f(p^r)$ be the order of the reduction of $f$ modulo $p^r$; equivalently, it is the period of the associated invertible recurrence.  
We say that $p$ is a \emph{$\WSS$ prime for $f$} if
\[
k_f(p^2)=k_f(p).
\]
This notation makes explicit that the equality depends on the polynomial, not merely on its discriminant.

The arithmetic setting of \cite{B,B2} is obtained as follows.  
Let $d>0$ be a fundamental discriminant and let $\varepsilon_d$ be the fundamental unit of $\Q(\sqrt d)$.  
Write $\varepsilon_d=(a+b\sqrt d)/2$, where $(a,b)$ is the minimal solution of $x^2-dy^2=\pm4$.  The corresponding recurrence is
\[
F_{n+2}=aF_{n+1}-\frac{a^2-db^2}{4}F_n,
\]
and its characteristic polynomial is the minimal polynomial
\[
P_d(X)=X^2-aX+\frac{a^2-db^2}{4}
\]
of $\varepsilon_d$.  Subject to the hypotheses stated in Theorem~\ref{periods and p-rationality}, the equality $k_{P_d}(p^2)=k_{P_d}(p)$ holds precisely when $\Q(\sqrt d)$ is \emph{not} $p$-rational.  
For $d=5$ this is the classical WSS problem, and no such prime is known.
\vspace{1em}

For a positive square-free integer $d$, we use the existential notation $\WSS(d)^*$: it means that there is a monic quadratic $f\in\Z[X]$ with unit constant term modulo $p$, positive discriminant with square-free part $d$, and $k_f(p^2)=k_f(p)$.  
No root of $f$ is required to be a fundamental unit.  
Consequently, statements about $p$-rationality are made only for the specific fundamental-unit recurrence above, whereas the inverse constructions in Section~\ref{sec:inv-prob} concern $\WSS(d)^*$.

With every periodic second-order recurrence one can associate a cyclic code of dimension two whose codewords are period windows of the recurrence as the initial conditions vary.
The weight distribution of cyclic codes of dimension two is studied over finite fields in \cite{SZS,V1,V2} and over rings in \cite{SHS,SLS}. The methods in \cite{SHS,SLS,SZS}
are based on a local Binet-like formula; while the technique in \cite{V1,V2} is a more classical coding theoretic approach.
Their salient property is that they are one- or two-weight codes, a fact which connects them with strongly regular graphs in the projective case \cite{BM}.  
Some of our codes are MDS in the reducible case.  
In the irreducible case, for $p\equiv3\pmod4$, we construct extremal NMDS codes with parameters $[2p+2,2,2p]_p$. 

The material is organized as follows.  
Section~2 contains preliminary notions.  
Section~3 recalls the precise connection between period equality and $p$-rationality.  
Section~4 constructs, for a fixed odd prime $p$, infinitely many pairwise distinct square-free integers $d$ for which $p$ is $\WSS(d)^*$.  
Section~5 compares the weight distributions of the associated codes over $\F_p$ and $\Z_{p^2}$.  
Section~6 concludes the article.
\section{Preliminary}

\subsection{Codes}

We assume that the reader has some basic familiarity with Coding Theory as can be gained from the first chapter of \cite{MacSlo} or \cite{HuffmanPless}.

The {\bf Hamming weight} of a vector $\vec{u}$ over a finite field $F$ and of length $n$ is the number of all nonzero coordinates of $\vec{u}$, denoted by $\wt(\vec{u})$.
A {\bf code} of length $n$ over a finite field $F$ is a vector subspace of $F^n.$
Its elements are called {\bf codewords}.
Its {\bf dimension} is denoted by $\kappa,$ and its minimum weight or {\bf minimum distance} by $d.$
The three parameters are written compactly as $[n,\kappa,d]$ or $[n,\kappa,d]_q$ if $F=\mathbb F_q$, the finite field of order $q,$ where $q=p^r$ ($p$ a prime and $r\ge1$).
An $[n,\kappa,n-\kappa+1]$ is called {\bf MDS}.
An $[n,\kappa,n-\kappa]$ is called almost-MDS, simply {\bf AMDS}
and near-MDS, simply {\bf NMDS} if its dual is also AMDS.
A {\bf generator matrix} $G$ of the code is any $\kappa$ by $n$ matrix whose rows span the code.

A code is {\bf projective} if any pair of columns of $G$ is linearly independent.
A code is {\bf two weight} if it has only two nonzero weights $w_1,\,w_2$ with $w_1 < w_2$.  This is denoted
by $[n,k,\{w_1,\,w_2\}].$ 
It was shown by Delsarte \cite{Delsarte} that if the code is projective then $w_2-w_1$ is a power of the
characteristic $p$ of $F.$
If  $m>1$ is the largest factor of $w_2-w_1$ that is coprime to $p$,
then by \cite{K} the code is an $m$-fold repetition of a projective code. (See \cite{K} for a precise definition of the $m$-fold repetition).

Considering codes over a finite commutative ring $R$ we call linear codes of length $n$ over $R$ the
submodules of $R^n.$ We use the notation $(n,M,d)$ for codes with $M$ the size and $d$ the minimum distance of the code.

\subsection{Cyclic codes}

A {\bf cyclic code} of length $n$
over a finite field $F$ is an ideal of the polynomial ring $F[x]/(x^n-1).$ It can be shown that this ideal is of the form $(g),$ for some monic divisor $g$ of $x^n-1.$ The polynomial $g$ is called the
{\bf generator polynomial}  and the  {\bf check polynomial} is defined as $\frac{x^n-1}{g}.$
Similarly a cyclic code of length $n$
over a finite ring $R$ is an ideal of the polynomial ring $R[x]/(x^n-1).$ For a free code we define similarly its generator and check polynomials.
Given an invertible second-order recurrence
$$F_{n}=-AF_{n-1}-BF_{n-2},\; n\ge 2,$$
of period $T$ over $R$, there is a length-$T$ cyclic code $C$ of rank two
constructed from the recurrence, with codewords
$$ (F_0,\dots,F_{T-1})$$
and initial conditions arbitrary. Thus
$$C=\{(F_0,\dots,F_{T-1}) \mid F_0,F_1 \in R\}.$$
The reciprocal of the characteristic polynomial is, up to normalization, the check polynomial of $C$.  
Replacing a polynomial by its reciprocal does not change its order, so throughout the paper we freely pass between these two descriptions when discussing periods.

In the present paper, the number of codewords of weight $w$ is denoted by $A_w$ both in the field and ring case. The sequence $(A_w)_{w=0}^n$ is called the {\bf weight distribution} of the code.

\subsection{Basic Number theory} \label{sub-number-theory}
The {\bf Legendre symbol} $(\frac{.}{p})$ is defined for an odd prime $p$ by the relation
$$\left(\frac{x}{p}\right)=
\begin{cases}
	1, & \text{if}\ x \in Q \\
	0 ,& \text{if}\ x=0\\
	-1, & \text{otherwise}
\end{cases}
$$
where $Q$ is the set of quadratic residues modulo $p$ defined by
$$Q=\{x^2 \mid x \in \F_p^\times\}.$$

In particular, a quadratic $t^2+ut+v$ is reducible over $\F_p$, for an odd
prime $p$, if and only if its {\bf discriminant}
$\Delta_p=u^2-4v \in \F_p$ satisfies
$(\frac{\Delta_p}{p})\in \{0,1\}$.

A {\bf number field} is a finite extension of the rationals. It is {\bf quadratic} if the degree of the extension is two. 
It is {\bf real quadratic} if it is of the form $\mathbb{Q}(\sqrt{d})$ for some square-free integer $d>1.$ The {\bf class number} of a number field $F$ is the order of its class group; that is the group of fractional ideals of $F$ modulo principal ideals. 
A field extension $F/K$ is said to be Galois if it is normal (every irreducible polynomial over $K$ that has one root in $F$ splits completely in $F$) and separable (all extensions of ${\Q}$ are separable). 
The Galois group of $F/L$ is the group of automorphisms of $F$ that fixes $K$ pointwise.

\section{Real quadratic \texorpdfstring{$p$}{p}-rational fields}\label{p-rational fields}

We recall only the arithmetic criterion needed later.  
Let $F$ be a number field, let $S_p$ be the set of primes of $F$ above $p$, and let $F_{S_p}$ be the maximal pro-$p$ extension of $F$ unramified outside $S_p$.  
The field $F$ is called $p$-rational if $\operatorname{Gal}(F_{S_p}/F)$ has the equivalent properties listed in \cite[Theorem and Definition~2.1]{Movahhedi-Nguyen}.  
For a real quadratic field and an unramified odd prime $p$, this condition is controlled by the class number and the image of a fundamental unit in the completions at $p$, see \cite[Proposition~4.1]{Greenberg}.

Let $d>0$ be a fundamental discriminant.  
Denote by $\varepsilon_d$ and $h_d$ the fundamental unit and class number of $\Q(\sqrt d)$, respectively, and let $\overline{\varepsilon}_d$ be the conjugate of $\varepsilon_d$.
The associated sequence is defined by $F_0=0$, $F_1=1$, and
\begin{equation}\label{recurrence for d}
	F_{n+2}=(\varepsilon_d+\overline{\varepsilon}_d)F_{n+1}
	-\varepsilon_d\overline{\varepsilon}_dF_n.
\end{equation}
Its characteristic polynomial is the minimal polynomial $P_d$ of $\varepsilon_d$.  Write $k_{P_d}(m)$ for its period modulo $m$.

\begin{thm}[\cite{B}]\label{periods and p-rationality}
	Let $d>0$ be a fundamental discriminant and let $p$ be an odd prime such
	that
	\[
	p\nmid(\varepsilon_d-\overline{\varepsilon}_d)^2h_d.
	\]
	Then
	\[
	k_{P_d}(p^2)=k_{P_d}(p)
	\quad\Longleftrightarrow\quad
	\Q(\sqrt d)\text{ is not $p$-rational}.
	\]
	Equivalently, $\Q(\sqrt d)$ is $p$-rational precisely when the two periods are different.
\end{thm}

For completeness, the local formulation underlying this theorem uses the normalized additive valuation $v$ on a completion $F_v$, with $v(p)>0$, and compares the valuation of the local $p$-adic logarithm of $\varepsilon_d$ with the valuation of a suitable power $\varepsilon_d^{Nv-1}-1$.  We do not use a congruence for a $p$-adic logarithm modulo $p^2$; the valuation statement, together with the hypotheses in Theorem~\ref{periods and p-rationality}, is the precise input.

The theorem applies only to the recurrence determined by the fundamental unit.  A different quadratic polynomial with the same square-free discriminant need not define the same period or imply anything about $p$-rationality.  This is why the more general notation $\WSS(d)^*$ used below is explicitly existential.

\section{Inverse problem}\label{sec:inv-prob}

Let $R_r=\Z/p^r\Z$.  
If $f\in R_r[x]$ is monic and $f(0)$ is a unit, its \emph{order} is the least positive integer $N$ such that $f\mid x^N-1$.  
Over a finite field this order is the least common multiple of the multiplicative orders of the roots of $f$ in a splitting field \cite[Theorem~8.27]{LN}.  
It also equals the period of the corresponding invertible recurrence \cite[Theorem~8.28]{LN}.

We first record a device that ensures that the square-free parts obtained below are genuinely distinct.

\begin{lem}\label{lem-prime-discriminants}
 1	Let $H(x)=x^2+a x+b\in R_2[x]$ be square-free modulo $p$, where $p$ is an odd prime.  
	Suppose that $H$ has order $N$ over $R_2$ and that its reduction has the same order $N$ over $R_1$.  
	Then there are infinitely many monic polynomials $f_\ell\in\Z[x]$, congruent to $H$ modulo $p^2$, whose discriminants are pairwise distinct positive primes $\ell$.  
	In particular, $p$ is $\WSS(d)^*$ for infinitely many distinct square-free integers $d$.
\end{lem}

\begin{proof}
	Choose integer representatives $a_0,b_0$ of $a,b$ (that is to say, $a=a_0$, $b=b_0$) and put
	$\Delta_0=a_0^2-4b_0$.  
	We claim that $p\nmid\Delta_0$. 
	For otherwise, consider the reduction polynomial $\overline{H} \equiv H \pmod{p}$. 
	Its discriminant is $a_0^2-4b_0 \equiv 0 \pmod{p}$. 
	Hence, $\overline{H}$ has square factors, which is impossible. 
	
	By the Chinese remainder theorem, the conditions
	\[
	\ell\equiv\Delta_0\pmod{p^2},\qquad \ell\equiv1\pmod4
	\]
	specify a reduced residue class modulo $4p^2$.  
	By Dirichlet's theorem, there exist infinitely many such primes $\ell$. 
	Choose an odd integer
	$A\equiv a_0\pmod{p^2}$ and set
	\[
	B=\frac{A^2-\ell}{4} \in \mathbb{Z},
	\]
	because $A^2 \equiv 1 \pmod{4}$ and $\ell \equiv 1 \pmod{4}$. 
	Moreover, we have $B\equiv b_0\pmod{p^2}$ since 
	\[ A^2-\ell \equiv a_0^2 - (a_0^2 - 4b_0) \equiv 4b_0 \pmod{p^2}. \]
	Thus $f_\ell(x)=x^2+Ax+B$ is congruent to $H$ modulo $p^2$ and has discriminant $\ell$.  
	Its reductions modulo $p$ and $p^2$ consequently have the same order $N$, as required.
\end{proof}

\subsection{The reducible construction}\label{sec-order-p-1}

Choose $r=1$ and an element $s\in\F_p^\times$ of order $p-1$, and put
\[
h(x)=(x-r)(x-s).
\]
The two factors are coprime, $h(0)\ne0$, and $h$ has order $p-1$.
The coprime-factor form of Hensel's lemma lifts the factorization $x^{p-1}-1=h(x)g(x)$ uniquely to a factorization modulo $p^2$,
\[
x^{p-1}-1=H(x)G(x),\qquad \overline H=h.
\]
The order of $H$ divides $p-1$, while reduction modulo $p$ shows that it cannot be smaller; hence $H$ also has order $p-1$.  
Lemma \ref{lem-prime-discriminants} now yields infinitely many distinct primes $\ell$ for which $p$ is $\WSS(\ell)^*$ and the common period is $p-1$.
Since the quotient of the two roots also has order $p-1$, the associated $[p-1,2,p-2]_p$ code is MDS.  
More generally, choosing distinct roots whose orders have least common multiple $D$ gives the same period conclusion, although the corresponding code need not be MDS unless the quotient of the roots has the required order.

\subsection{The double-root case}

The preceding Hensel argument does not apply to $(x-1)^2$, because this factor is not coprime to its complementary factor in $x^p-1=(x-1)^p$ over $\F_p$.  In fact the proposed quadratic lift does not exist for $p\ge5$.

\begin{prop}\label{prop-no-double-lift}
	If $p\ge5$, there is no monic quadratic divisor $H$ of $x^p-1$ in
	$(\Z/p^2\Z)[x]$ satisfying $\overline H=(x-1)^2$.
\end{prop}

\begin{proof}
	Suppose $x^p-1=H(x)Q(x)$, where $Q$ is monic.  
	Modulo $p$ one has
	$\overline Q=(x-1)^{p-2}$.  Since $p\ge5$, both $\overline H$ and $\overline Q$ vanish to order at least two at $1$.  Therefore
	\[
	H(1),H'(1),Q(1),Q'(1)\in p(\Z/p^2\Z).
	\]
	Differentiating and evaluating at $1$ gives
	\[
	p=H'(1)Q(1)+H(1)Q'(1)=0\quad\text{in }\Z/p^2\Z,
	\]
	a contradiction.
\end{proof}

The exceptional prime $p=3$ does admit such a factor:
\[
x^3-1=(x-1)(x^2+x+1)
\]
over $\Z$, and $x^2+x+1$ reduces to $(x-1)^2$ modulo $3$ and has order $3$ modulo both $3$ and $9$.  
For example, $f(x)=x^2+10x+1$ is congruent to it modulo $9$ and has discriminant $96=6\cdot4^2$.  Hence $3$ is $\WSS(6)^*$ for this specified polynomial.

\subsection{The irreducible construction}\label{subsec-irr-cyc}

\begin{prop}\label{prop-order-2p+2}
	For every odd prime $p$ there is an irreducible quadratic
	$h\in\F_p[x]$ of exact order $2(p+1)$.
\end{prop}

\begin{proof}
	Let $g$ be the generator of the cyclic group $\F_{p^2}^{\times}$ and set
	\[
	\alpha=g^{(p-1)/2}.
	\]
	Then
	\begin{align*}
		\ord(\alpha)
		&=\frac{p^2-1}{\gcd(p^2-1,(p-1)/2)}=2(p+1),\\
		\alpha^{p+1}&=g^{(p^2-1)/2}=-1.
	\end{align*}
	Moreover, $\alpha\notin\F_p$, since otherwise its order would divide
	$p-1$.  Consequently
	\[
	h(x)=(x-\alpha)(x-\alpha^p) =x^2-(\alpha+\alpha^p)x-1
	\]
	belongs to $\F_p[x]$, is irreducible, and has exact order $2(p+1)$.
\end{proof}

Since $p\nmid 2(p+1)$, the polynomial $x^{2(p+1)}-1$ is square-free modulo $p$.  
Then by Hensel's lift, there exists a quadratic polynomial $H \mid x^{2(p+1)}-1 \in \mathbb{Z}_{p^2}[x]$ such that $\overline H=h$.  
The order of $H$ is exactly $2(p+1)$. 
For otherwise, $h$ will have order smaller than $2(p+1)$, a contradiction. 
Applying Lemma~\ref{lem-prime-discriminants}, we have 

\begin{coro}
	There are infinitely many distinct primes $\ell$ such that $p$ is $\WSS(\ell)^*$ with common period $2(p+1)$. 
\end{coro}

\begin{ex}
	When $p=13$, let $H = x^2 + 167x + 168 \in \mathbb{Z}_{p^2}[x]$ and $h = \overline{H} = x^2 + 11x +12 \in \mathbb{F}_{p}[x]$. 
	The cyclic code $C_1$ has parameter $[28,2,\{24,28\}]_{13}$. 
	And therefore it is a fourfold repetition of a $[7,2,6]_{13}$ MDS code. 
\end{ex}

\section{Comparing weight distributions}

Let $h(x) \mid x^n-1$ be the check polynomial of the free cyclic code $C_2$ of length $n$ and dimension $2$ over $\mathbb{Z}_{p^2}$, where $\gcd(n,p)=1$.
Let $C_1$ be the cyclic code of length $n$ and dimension $2$ over $\mathbb{Z}_p$ with the check polynomial $\bar{h} \equiv h \pmod{p}$, i.e., $h$ is the Hensel lift of $\bar{h}$.

It is known that the weight enumerator of an MDS code over $\mathbb{F}_q$ is unique (see, e.g., \cite[Ch.~11, \S 3, Theorem~6]{MacSlo}), which is given by $A_0 = 1$, $A_w = 0$ for $0 < w < d$, and
\[
A_w = \binom{n}{w} (q-1) \sum_{j=0}^{w-d} (-1)^j \binom{w-1}{j} q^{w-d-j}.
\]
We require some facts on the algebraic structure of cyclic codes over $\mathbb{Z}_{p^2}.$
\begin{prop}[\texorpdfstring{\cite[Proposition 2]{WS-DM}}{}]
	Let $C$ be a cyclic code of length $n$ over $\mathbb{Z}_{p^2}$ with $\gcd(n,p) = 1$.
	Then there exist polynomials $f_1, f_2, f_3, \in \mathbb{Z}_{p^2}[x]$ such that $f_1 f_2 f_3 = x^n-1$ and $C = \langle f_1 f_2, pf_1 \rangle$.
	Moreover, we have $pC = \langle pf_1 \rangle$.
	In particular, if $C$ is free, then $f_2 = 1$.
\end{prop}
We are now ready for the main result of this section.
\begin{thm}\label{thm-wd}
	Assume that $C_1$ is an $[n,2,n-1]_p$ MDS code.  Then
	\[
	A_{n-1}(C_1)=n(p-1),\qquad
	A_n(C_1)=(p-1)(p-n+1),
	\]
	and
	\[
	A_{n-1}(C_2)=n(p^2-1),\qquad
	A_n(C_2)=(p^2-1)(p^2-n+1).
	\]
	In particular, these are the only nonzero weights of $C_2$.
\end{thm}
\begin{proof}
	Since $C_2$ is a free cyclic code of length $n$ and dimension $2$ over $\mathbb{Z}_{p^2}$, it has a monic generator polynomial $g(x)\in\mathbb{Z}_{p^2}[x]$ of degree $n-2$, and
	\[
	C_2=\langle g(x)\rangle.
	\]
	Hence $g^{(0)}:=g(x)$ and $g^{(1)}:=xg(x)$ form a $\mathbb{Z}_{p^2}$-basis of $C_2$. 
	Let $\bar g(x)$ be the reduction of $g(x)$ modulo $p$. 
	Since $h$ is the Hensel lift of $\bar h$, the reduction $\bar g(x)$ is the monic generator polynomial of $C_1$ of degree $n-2$, so
	\[
	C_1=\langle \bar g(x)\rangle.
	\]
	Therefore $\bar g^{(0)}:=\bar g(x)$ and
	$\bar g^{(1)}:=x\bar g(x)$ form an $\mathbb{F}_p$-basis of $C_1$.
	
	Fix this basis. For $\vec{c}\in C_2$, write its coordinates with respect to this basis uniquely as
	\[
	\vec{c}=u_0 g^{(0)}+u_1 g^{(1)},\qquad u_i\in\mathbb{Z}_{p^2}.
	\]
	For each $i$, write $u_i=\bar u_i+p v_i$ with $\bar u_i,v_i\in\{0,1,\ldots,p-1\}$. Then
	\[
	\vec{c}=\vec{c}_1'+p \vec{c}_2',
	\]
	where
	\[
	\vec{c}_1'=\bar u_0 g^{(0)}+\bar u_1 g^{(1)},\qquad
	\vec{c}_2'=v_0 g^{(0)}+v_1 g^{(1)}.
	\]
	Let
	\[
	\vec{c}_1=\bar u_0 \bar g^{(0)}+\bar u_1 \bar g^{(1)},\qquad
	\vec{c}_2=v_0 \bar g^{(0)}+v_1 \bar g^{(1)}
	\]
	be the corresponding elements of $C_1$. 
	The map $\vec{c}\mapsto(\vec{c}_1,\vec{c}_2)$ is a bijection $C_2\to C_1\times C_1$. 
	Indeed, the coefficients $u_0,u_1\in\mathbb{Z}_{p^2}$ of $\vec{c}$ are uniquely determined, and each $u_i$ has a unique decomposition $u_i=\bar u_i+p v_i$ with $\bar u_i,v_i\in\{0,\ldots,p-1\}$. 
	The pairs $(\bar u_0,\bar u_1)$ and $(v_0,v_1)$ uniquely determine $\vec{c}_1$ and $\vec{c}_2$, respectively, and conversely. 
	Hence the map is bijective. 
	
	Since $C_1$ is MDS, its only nonzero weights are $n-1$ and $n$, and the weight enumerator gives
	\[
	A_{n-1}(C_1)=n(p-1),\qquad
	A_n(C_1)=(p-1)(p-n+1).
	\]
	
	We now determine the weights of $C_2$. Let $\vec{c}\in C_2$ correspond to $(\vec{c}_1,\vec{c}_2)$ as above. 
	If $\mathrm{wt}(\vec{c}_1)=n$, then every coordinate of $\vec{c}$ is nonzero modulo $p$, hence nonzero in $\mathbb{Z}_{p^2}$, thus $\mathrm{wt}(\vec{c})=n$.
	
	Suppose $\mathrm{wt}(\vec{c}_1)=n-1$, and let $i_0$ be its unique zero coordinate. 
	For every $j\neq i_0$, the coordinate $(\vec{c}_1)_j$ is nonzero, so $\vec{c}_j\neq 0$. At $i_0$, we have
	$(\vec{c}_1')_{i_0}=p x_0$ for some $x_0\in\{0,\ldots,p-1\}$, and
	\[
	\vec{c}_{i_0}=p\bigl(x_0+(\vec{c}_2')_{i_0}\bigr).
	\]
	Thus $\vec{c}_{i_0}=0$ if and only if $(\vec{c}_2)_{i_0}=-x_0$ in $\mathbb{F}_p$.
	Next, consider the projective map
	\[
	C_1\to\mathbb{F}_p,\qquad \vec{c}_2\mapsto (\vec{c}_2)_{i_0}.
	\]
	Since $C_1$ is an $[n,2,n-1]_p$ MDS code, its projection onto coordinate $i_0$ is not the zero map. 
	Because $C_1$ is linear, this projection is surjective onto $\mathbb{F}_p$, and every element of $\mathbb{F}_p$ appears exactly $p$ times as the $i_0$-th coordinate of codewords of $C_1$. 
	Therefore, for each $\vec{c}_1$ of weight $n-1$, there are exactly $p$ of the $p^2$ possible $\vec{c}_2$'s such that $\mathrm{wt}(\vec{c})=n-1$.
	Consequently, the contribution from $\vec{c}_1\neq0$ to
	$A_{n-1}(C_2)$ is
	\[
	p A_{n-1}(C_1)=p\,n(p-1).
	\]
	
	It remains to account for the kernel of reduction, namely $pC_2$. 
	If $\vec{c}_1=0$, then $\vec{c}=p \vec{c}_2'$, and $\mathrm{wt}(\vec{c})=\mathrm{wt}(\vec{c}_2)$ because multiplication by $p$ preserves zero and nonzero coordinates modulo $p^2$. Hence $pC_2$ has the same nonzero weight distribution as $C_1$. In particular, its contribution to weight $n-1$ is
	\[
	A_{n-1}(C_1)=n(p-1).
	\]
	Therefore
	\[
	A_{n-1}(C_2)=p\,n(p-1)+n(p-1)=n(p^2-1).
	\]
	Since $|C_2|=p^4$, subtracting the zero codeword and the codewords of weight $n-1$ gives
	\[
	A_n(C_2)=p^4-1-A_{n-1}(C_2)
	=p^4-1-n(p^2-1)
	=(p^2-1)(p^2-n+1).
	\]
	Finally, the classification above shows that every nonzero codeword of $C_2$ has weight either $n-1$ or $n$. Hence these are the only nonzero weights of $C_2$.
\end{proof}

\begin{ex}{\rm
		By Magma \cite{magma}, we have found the polynomials such that $C_1$ and $C_2$ have only two Hamming weights $w_1, w_2$ $(w_1 < w_2)$.
		The parameters of such codes are listed in Table \ref{tab-two-weight}, where $a_1 = A_{w_1}(C_1)$, $a_2 = A_{w_2}(C_1)$, and $A_1 = A_{w_1}(C_2)$, $A_2 = A_{w_2}(C_2)$.
		Moreover, ``IRR'' means that the check polynomial is irreducible, and ``RED'' means that the check polynomial is reducible.
		
		\begin{table}[htbp]
			\caption{Cyclic $[n, 2, \{w_1,w_2\}]_p$ codes for some $\WSS(d)$ primes $p$ from \cite[Table, p.~109]{B}} \label{tab-two-weight}
			\resizebox{\textwidth}{!}{%
				\begin{tabular}{lllll}
					\toprule
					$(n,d,p)$ & $\{w_1,w_2\}$ & $(a_1,a_2)$ & $(A_1,A_2)$ & poly \\
					\toprule
					$(28,8,13)$ & \{24,28\} & (84,84) & (1176,27384) & IRR  \\
					$(30,8,31)$ & \{29,30\} & (900,60) & (28800,894720) & RED  \\
					$(104,12,103)$ & \{102,104\} & (5304,5304) & (551616,111999264) & IRR   \\
					$(484,13,241)$ & \{480,484\} & (29040,29040) & $(7027680,3366374880)$ & IRR   \\
					$(8,24,7)$ & \{6,8\} & (24,24) & (192,2208) & IRR \\
					$(174,24,523)$ & \{172,174\} & (45414,228114) & $(23796936,74794316904)$ & RED  \\
					\bottomrule
				\end{tabular}%
			}
		\end{table}
	}
\end{ex}

\begin{coro}
	Keep the above notations.
	Then we have $A_{w_i}(C_1)$ divides $A_{w_i}(C_2)$ for $i=1,2$, when $n = p-1$ and $n = (p+1)/2$.
\end{coro}
\begin{proof}
	For each $n$, we always have $A_{n-1}(C_1)$ divides $A_{n-1}(C_2)$.
	As for $A_n$, it is clear that $A_{n}(C_1)$ divides $A_{n}(C_2)$ if and only if
	\[ (p-n+1) \mid (p+1)(p^2 - n + 1) \]
	which is true when $n = p-1$ and $n = (p+1)/2$.
\end{proof}

If $m>1$ is the largest factor of $w_2-w_1$ that is coprime to $p$, then by \cite{K} the code is an $m$-fold repetition of a projective code. 

\begin{coro}\label{cor-block-lift}
	Let $n=m\ell$.  Suppose, after a coordinate permutation, that the reduction $C_1$ is a monomial $\ell$-fold repetition of an $[m,2,m-1]_p$ MDS code.
	Assume moreover that every word $c\in C_2$ satisfies the block relation
	\[
	c_{j+tm}=\widetilde\rho^{,t}c_j
	\qquad(0\le j<m,\ 0\le t<\ell)
	\]
	for a fixed unit $\widetilde\rho\in\Z_{p^2}^{\times}$.  
	Then $C_2$ has nonzero weights $n-\ell$ and $n$, with
	\[
	A_{n-\ell}(C_2)=\frac{n}{\ell}(p^2-1),\qquad
	A_n(C_2)=p^4-1-A_{n-\ell}(C_2).
	\]
\end{coro}

\begin{proof}
	The unit block relation shows that a zero in one block occurs in the corresponding coordinate of every block. Consequently, the Hamming weight of a lifted codeword is $\ell$ times the weight of its first block.
	As in the proof of Theorem~\ref{thm-wd}, for the first $m$ coordinates, there are
	\[
	m(p^2-1)=\frac{n}{\ell}(p^2-1)
	\]
	codewords with precisely one zero position in the first block, and every other nonzero word has full weight. Therefore, the corresponding lifted words have weight $n-\ell$ in the first case and weight $n$ in the second case. This gives the stated weight enumerator.
\end{proof}

In Tables~\ref{tab-two-weight-2},~\ref{tab-two-weight-3}, we display selected cyclic $[n, 2, w_1]_p$ codes over $\mathbb F_p$ with one nonzero weight $w_1$ or two nonzero weights $w_1$ and $w_2$ $(w_1 < w_2)$ for $\WSS(d)^*$ primes $p$ $(7 \le p \le 167)$. 
In the first column, $(n,d,(\frac{d}{p}),p)$ records the code length, the square-free part $d$ of the positive integral discriminant, its Legendre symbol, and the prime.
For each $p$ only one admissible value of $d$ is displayed.
The second and third columns are the same as those in Table~\ref{tab-two-weight}.
In the fourth column, we give a check polynomial $h(x)$ over $\mathbb F_p$ of order $p-1$ or $2p+2$ and its Hensel lift $H(x)$ over $\mathbb Z_{p^2}$.
In the fifth column, we distinguish the codes with MDS, NMDS, or the direct sum of the 2-fold of an $[n/2, 2, w_1/2]_p$ MDS (simplex) code denoted by 2-MDS or the 4-fold of an $[n/4, 2, w_1/4]_p$ MDS (simplex) code denoted by 4-MDS.
More details can be seen in the database \url{https://new-cms.sogang.ac.kr/cicagolab/9768/subview.do#wss-31}.

\medskip

The tables are explained by the following direct trace calculation.

\begin{prop}\label{prop-trace-weights}
	Let $\alpha$ be the element of order $n=2(p+1)$ constructed in Proposition~\ref{prop-order-2p+2}, and let $C_1$ be the corresponding dimension-two cyclic code.  
	If $p\equiv3\pmod4$, then every nonzero word has weight $2p$, so $C_1$ is an NMDS code with parameters $[2(p+1),2,2p]_p$.  
	If $p\equiv1\pmod4$, then the nonzero weights are $2p-2$ and $2p+2$, each with frequency $(p^2-1)/2$.
\end{prop}

\begin{proof}
	Up to reversing the coordinate order, the code has the trace representation
	\[
	c_x=\bigl(\Tr(x),\Tr(x\alpha),\ldots,
	\Tr(x\alpha^{n-1})\bigr),\qquad x\in\F_{p^2},
	\]
	where $\Tr=\Tr_{\F_{p^2}/\F_p}$.   
	For $x\ne 0$, a coordinate is zero precisely when
	\[
	\alpha^j\in x^{-1} \cdot \ker(\Tr)^{\times}, 
	\]
	where $\ker(\Tr)^{\times} = \ker(\Tr) \setminus \{0\}$ is a coset of $\F_p^\times$ in
	$\F_{p^2}^\times$.  
	For such $\alpha$, every nonempty intersection of a coset of $\F_p^\times$ with $\langle\alpha\rangle$ therefore has size
	\[
	|\langle\alpha\rangle \cap \F_p^\times|=\gcd(2p+2,p-1)=\gcd(4,p-1).
	\]
	
	If $p\equiv3\pmod4$, this size is $2$ and 
	\[ |\langle\alpha\rangle \cdot \F_p^\times|= \frac{|\langle\alpha\rangle| \cdot |\F_p^\times|}{|\langle\alpha\rangle \cap \F_p^\times|} = p^2-1, \]
	so every nonzero word has exactly two zero coordinates and weight $2p$.  
	Moreover, $\alpha^{p+1}=-1$, and then $c_{j+p+1}=-c_j$.  
	Thus $C_1$ is monomially equivalent to a twofold repetition of the $[p+1,2,p]_p$ simplex code.  
	Its dual contains a word of weight two and no word of weight one.  
	Hence both $C_1$ and its dual are AMDS, so $C_1$ is NMDS.
	
	If $p\equiv1\pmod4$, the intersection size is $4$ and $|\langle\alpha\rangle \cdot \F_p^\times|$ has size $(p^2-1)/2$.  Exactly half of the $p^2-1$ nonzero values of $x$ give four zero coordinates, while the other half give none.  This yields the stated weights and frequencies.
\end{proof}
\begin{rem}{\rm
		The length $2p+2$ is the largest possible for a near MDS code of dimension $2$ over $\F_p,$ by Proposition 5.1. of \cite{DL}. Such codes are called extremal NMDS codes.
	}
\end{rem}

\begin{rem}{\rm
		In the fourth column of Tables~\ref{tab-two-weight-2},~\ref{tab-two-weight-3}, $h(x)$ can be distinguished by two quadratic polynomials $h_{+1}(x)$ and $h_{-1}(x)$ depending on the value of
		$(\frac{\Delta_p}{p})= \pm 1$ for a fixed prime $p$. We observe that $h_{-1}(0) = h_{+1}(1) -1$.
		
		The reason is as follows.
		Let $h_{-1}(x)= (x- \alpha)(x- \alpha^p)$ as in Proposition~\ref{prop-order-2p+2}, where  $\alpha \in \F_{p^2}\setminus \F_p$ satisfies $\alpha^{p+1}=-1$. Then $h_{-1}(x)=x^2-(\alpha+\alpha^p)x+\alpha^{p+1}=x^2-(\alpha+\alpha^p)x-1$. Thus, $h_{-1}(0)=p-1$. On the other hand, let $h_{+1}(x) = (x-1)(x-\beta)$ where $\beta \in \F_p$ has order $p-1$ by the procedure described in Section~\ref{sec-order-p-1}. By expanding $h_{+1}$, we have $h_{+1}(1)-1=p-1$. Therefore, $h_{-1}(0) = h_{+1}(1)-1$.
	}
\end{rem}

\begin{table}[htbp]
	\centering
	\caption{Cyclic $[n, 2, \{w_1, w_2\}]_p$ codes for some $\WSS(d)^*$ primes $p$ ($7 \le p \le 47$)} \label{tab-two-weight-2}
	\resizebox{\textwidth}{!}{%
		\begin{tabular}{lllll}
			\toprule
			$(n, d, (\frac{d}{p}) ,p)$ & $\{w_1,w_2\}$  or  & $(a_1,a_2)$ or  & $h(x)$ and & MDS/NMDS or \\
			&  $\{w_1\}$ & $(a_1)$ & $H(x)$ & 2-MDS, 4-MDS\\
			\toprule
			$(6,85,+1,7)$ & $\{5,6\}$ & (36,12) &   $x^2 +x +5$ & MDS\\
			& & & $x^2+29x+19$  &  \\
			$(16,649,-1,7)$ & $\{14\}$ & (48) &  $x^2 +x +6$  & NMDS, 2-MDS \\
			& & & $x^2+29x+48$  &  \\
			\hline
			$(10,3,+1,11)$ & $\{9,10\}$ & (100,20) &  $x^2 +4x +6$  & MDS \\
			& & & $ x^2 + 26x + 94$  &  \\
			$(24,6,-1,11)$ & $\{22\}$ & (120) &  $x^2 +2x +10$ & NMDS, 2-MDS \\
			& & & $ x^2 + 24x + 120$  &  \\
			\hline
			
			
			
			$(16,426,+1,17)$ & $\{15,16\}$ & (256,32) &  $x^2 +10x +6$ & MDS\\
			& & & $x^2 + 248x + 40$   &  \\
			
			$(36,193,-1,17)$ & $\{32,36\}$ & (144,144) &  $x^2 +x +16$ & 4-MDS\\
			& & & $x^2 + 103x + 288$   &  \\
			\hline


			$(40,43081,-1,19)$ & $\{38\}$ & (360) &  $x^2 +2x +18$ & NMDS, 2-MDS \\
			& & & $x^2 + 211x + 360$ & \\
			\hline
			
			$(22,381,+1,23)$ & $\{21,22\}$ & (484,44) &  $x^2 +11x +11$ & MDS\\
			& & & $x^2 + 333x + 195$ & \\
			
			$(48,697,-1,23)$ & $\{46\}$ & (528) &  $x^2 +x +22$ & NMDS, 2-MDS\\
			& & & $x^2 + 70x + 528$ & \\
			
			\hline
			$(28,6821,+1,29)$ & $\{27,28\}$ & (784,56) &  $x^2 +9x +19$ & MDS\\
			& & & $x^2 + 415x + 425$ & \\
			
			$(60,46,-1,29)$ & $\{56,60\}$ & (420,420) &  $x^2 +6x +28$ & 4-MDS\\
			& & & $x^2 + 64x + 840$ & \\
			
			\hline
			$(30,915,+1,31)$ & $\{29,30\}$ & (900,60) &  $x^2 +19x +11$ & MDS\\
			& & & $x^2 + 546x + 414$ & \\
			
			$(64,384289,-1,31)$ & $\{62\}$ & (960) &  $x^2 +3x +30$ & NMDS, 2-MDS\\
			& & & $x^2 + 623x + 960$ & \\
			
			\hline
			$(36,7619,+1,37)$ & $\{35,36\}$ & (1296,72) &  $x^2 +23x +13$ & MDS\\
			& & & $ x^2 + 874x + 494 $ & \\
			
			$(76,29321,-1,37)$ & $\{72,76\}$ & (684,684) &  $x^2 +x +36$ & 4-MDS\\
			& & & $x^2 + 519x + 1368$ & \\
			
			\hline
			$(40,1369205,+1,41)$ & $\{39,40\}$ & (1600,80) &  $x^2 +23x +17$ & MDS\\
			& & & $x^2 + 1171x + 509$ & \\
			
			$(84,1523449,-1,41)$ & $\{80,84\}$ & (840,840) &  $x^2 +7x +40$ & 4-MDS\\
			& & & $x^2 + 1237x + 1680$ & \\
			
			\hline
			$(42,2022,+1,43)$ & $\{41,42\}$ & (1764,84) &  $x^2 +13x +29$ & MDS\\
			& & & $x^2 + 1260x + 588$ & \\
			
			$(88,125563,-1,43)$ & $\{86\}$ & (1848) &  $x^2 +x +42$ & NMDS, 2-MDS \\
			& & & $x^2 + 1420x + 1848$ & \\
			
			\hline
			$(46,2085,+1,47)$ & $\{45,46\}$ & (2116,92) &  $x^2 +35x +11$ & MDS\\
			& & & $x^2 + 599x + 1609$ & \\
			
			$(96,2554369,-1,47)$ & $\{94\}$ & (2208) &  $x^2 + 3x +46$ & NMDS, 2-MDS \\
			& & & $x^2 + 1601x + 2208$ & \\
			
			\bottomrule
		\end{tabular}%
	}
\end{table}

\begin{table}[htbp]
	\centering
	\caption{Cyclic $[n, 2, \{w_1, w_2\}]_p$ codes for some $\WSS(d)^*$ primes $p$ ($53 \le p \le 97$)} \label{tab-two-weight-3}
	\resizebox{\textwidth}{!}{%
		\begin{tabular}{lllll}
			\toprule
			$(n, d, (\frac{d}{p}), p)$ & $\{w_1,w_2\}$  or  & $(a_1,a_2)$ or  & $h(x)$ and & MDS/NMDS or \\
			&  $\{w_1\}$ & $(a_1)$ & $H(x)$ & 2-MDS, 4-MDS\\
			\toprule
			
			$(52,945867,+1,53)$ & $\{51,52\}$ & (2704,104) &  $x^2 +38x +14$ & MDS\\
			& & &  $x^2 + 1946x + 862$  &\\
			
			$(108,200593,-1,53)$ & $\{104,108\}$ & (1404, 1404) &  $x^2 + x +52$ & 4-MDS\\
			& & & $x^2 + 902x + 2808$ & \\
			
			\hline
			$(58,1849301,+1,59)$ & $\{57,58\}$ & (3364,116) &  $x^2 +6x +52$ & MDS\\
			& & &  $x^2 + 1363x + 2117$ & \\
			$(120, 5895841,-1,59)$ & $\{118\}$ & (3480) &  $x^2 + 12x +58$ & NMDS, 2-MDS \\
			& & & $x^2 + 2431x + 3480$ & \\
			
			\hline
			$(60,1210683,+1,61)$ & $\{59,60\}$ & (3600,120) &  $x^2 +6x +54$ & MDS\\
			& & & $x^2 + 2202x + 1518$ & \\
			
			$(124,782295,-1,61)$ & $\{120, 124\}$ & (1860, 1860) &  $x^2 + 2x +60$ & 4-MDS \\
			& & & $x^2 + 3540x + 3720$ & \\
			
			\hline
			$(66,257765,+1,67)$ & $\{65,66\}$ & (4356,132) &  $x^2 +53x +13$ & MDS\\
			& & & $x^2 + 1527x + 2961$ & \\
			
			$(136,2916193,-1,67)$ & $\{134\}$ & (4488) &  $x^2 + x +66$ & NMDS, 2-MDS \\
			& & & $x^2 + 3418x + 4488$ & \\
			
			\hline
			
			$(70,23218,+1,71)$ & $\{69,70\}$ & (4900,140) &  $x^2 +57x +13$ & MDS\\
			& & & $x^2 + 1832x + 3208$ & \\
			
			$(144,2429065,-1,71)$ & $\{142\}$ & (5040) &  $x^2 + 3x +70$ & NMDS, 2-MDS \\
			& & & $x^2 + 1565x + 5040$ & \\
			\hline
			
			$(72,917,+1,73)$ & $\{71,72\}$ & (5184,144) &  $x^2 +39x +33$ & MDS\\
			& & & $x^2 + 769x + 4559$ & \\
			
			$(148,27614737,-1,73)$ & $\{144, 148\}$ & (2664, 2664) &  $x^2 + x +72$ & 4-MDS\\
			& & & $x^2 + 5257x + 5328$ & \\
			\hline
			$(78,556413,+1,79)$ & $\{77,78\}$ & (6084,156) &  $x^2 +18x +60$ & MDS\\
			& & & $x^2 + 3731x + 2509$ & \\

			$(160,8784985,-1,79)$ & $\{158\}$ & (6240) &  $x^2 + 5x +78$  & NMDS, 2-MDS \\
			& & & $x^2 + 5930x + 6240$ & \\
			\hline
			$(82,8378,+1,83)$ & $\{81,82\}$ & (6724,164) &  $x^2 +68x +14$ & MDS\\
			& & & $x^2 + 400x + 6488$ & \\
			
			$(168,201961,-1,83)$ & $\{166\}$ & (6888) &  $x^2 + x +82$ & NMDS, 2-MDS \\
			& & & $x^2 + 914x + 6888$ & \\
			\hline
			$(88,4121103,+1,89)$ & $\{87,88\}$ & (7744,176) &  $x^2 +57x +31$ & MDS\\
			& & & $x^2 + 4062x + 3858$ & \\
			
			$(180,2638645,-1,89)$ & $\{176,180\}$ & (3960, 3960) &  $x^2 + 3x +88$ & 4-MDS\\
			& & & $x^2 + 6500x + 7920$ & \\
			\hline
			$(96,28355,+1,97)$ & $\{95,96\}$ & (9216,192) &  $x^2 +67x +29$ & MDS\\
			& & & $x^2 + 4044x + 5364$ & \\
			
			$(196,18186729,-1,97)$ & $\{192, 196\}$ & (4704, 4704) &  $x^2 + x +96$ & 4-MDS\\
			& & & $x^2 + 4269x + 9408$ & \\
			\bottomrule

		\end{tabular}%
	}
\end{table}

In what follows, we describe the weight distribution of a further reducible family over $\mathbb{Z}_{p^2}$.  
Let $a,b\in\mathbb{Z}_{p^2}^{\times}$ with $a\not\equiv b\pmod p$, and let $n$ be a common multiple of $\ord_{p^2}(a)$ and $\ord_{p^2}(b)$.  Then
\[
h(x)=\frac{(1-ax)(1-bx)}{ab}
=(x-a^{-1})(x-b^{-1})
\]
divides $x^n-1$ in $\mathbb{Z}_{p^2}[x]$ and defines a free cyclic code $C(a,b)$ of rank two.  Here $\ord_{p^r}$ always denotes multiplicative order in $(\mathbb{Z}/p^r\mathbb{Z})^{\times}$.
\begin{prop}\cite[Theorem 1]{SHS}
	Assume that $a \not\equiv b \pmod{p}$ and $e = \mathrm{ord}_p (b/a) = \mathrm{ord}_{p^2} (b/a)$.
	Then the code $C(a,b)$ has nonzero Hamming weights $\{n - \frac{n}{e}, n\}$, and their respective frequencies are
	\[ A_{n-\frac{n}{e}}(C(a,b)) = e(p^2 - 1), \quad A_n(C(a,b)) = (p^2 - 1)(p^2 - e + 1). \]
\end{prop}

\begin{coro}
	Assume that $a \not\equiv b \pmod{p}$ and $e = \mathrm{ord}_p (b/a) = \mathrm{ord}_{p^2} (b/a)$.
	Let $C_2 = C(a,b)$ and $C_1$ be the cyclic code over $\mathbb{F}_p$ with the check polynomial $\bar{h} \equiv h \pmod{p}$.
	Then $C_1$ has nonzero Hamming weights $\{n - \frac{n}{e}, n\}$, and their respective frequencies are
	\[ A_{n-\frac{n}{e}}(C_1) = e(p - 1), \quad A_n(C_1) = (p - 1)(p - e + 1). \]
	Moreover, the lower-weight frequency of $C_1$ always divides that of $C_2$, whereas $A_n(C_1)$ divides $A_n(C_2)$ if and only if
	\[ (p-e+1) \mid (p+1)(p^2 - e + 1). \]
\end{coro}
\begin{proof}
	The corresponding field version of the same reducible construction has nonzero weights $\{n - \frac{n}{e}, n\}$.
	The frequencies	of the weights can then be computed by the first two Pless power moments \cite[\S 7.3]{HuffmanPless}.
	They are the solutions of the system
	\begin{align*}
		A_{n-\frac{n}{e}}(C_1)+A_n(C_1)&=p^2-1,\\
		\left(n-\frac{n}{e}\right)A_{n-\frac{n}{e}}(C_1)
		+nA_n(C_1)&=pn(p-1).
	\end{align*}
	The ratio of the two lower-weight frequencies is $p+1$, so $A_{n-\frac{n}{e}}(C_1)\mid A_{n-\frac{n}{e}}(C_2)$ always.  For the full-weight frequencies the ratio is
	\[
	\frac{A_n(C_2)}{A_n(C_1)}
	=\frac{(p+1)(p^2-e+1)}{p-e+1}.
	\]
	Thus $A_n(C_1)\mid A_n(C_2)$ precisely when $(p-e+1)\mid(p+1)(p^2-e+1)$.
\end{proof}

\begin{rem}{\rm
		The code $C_2$ is projective if and only if $\mathrm{ord}_{p^2} (b/a) \geqslant n$.
		If $e = n$, then $C_2$ is MRD in the sense of \cite[Chap. 12]{SAS}, and the related $C_1$ is MDS.
	}
\end{rem}

\section{Conclusion and open problem}
We have studied period equality modulo $p$ and $p^2$ for specified invertible second-order recurrences and the associated dimension-two cyclic codes.  
For the recurrence defined by a fundamental unit, and only under the hypotheses of Theorem~\ref{periods and p-rationality}, period equality is equivalent to failure of $p$-rationality of the real quadratic field.  
For a general polynomial the notation $\WSS(d)^*$ records only an existential period property; it carries no automatic $p$-rationality conclusion.

The corrected inverse construction yields infinitely many pairwise distinct prime discriminants.  
The repeated-root construction survives only for $p=3$, while the reducible square-free and irreducible constructions work by ordinary Hensel lifting.  
We also obtained intrinsic weight counts over $\Z_{p^2}$ and a direct trace proof of the extremal NMDS family.  
A natural open problem is to determine which additional coding-theoretic restrictions can occur when the polynomial is specifically the minimal polynomial of a fundamental unit.

\end{document}